\documentclass[11pt]{article}
\usepackage[dvips,letterpaper,margin=1in]{geometry}
\usepackage[utf8]{inputenc}
\usepackage{amsmath}
\usepackage{amssymb}
\usepackage{amstext}
\usepackage{amsthm}
\usepackage[linesnumbered,ruled,vlined]{algorithm2e}
\usepackage{algorithmic}
\usepackage{fullwidth}
\usepackage{float}
\usepackage{colonequals}
\usepackage{hyperref}
\usepackage{comment}
\usepackage{color}
\usepackage{fancyhdr}

\newtheorem{theorem}{Theorem}[section]
\newtheorem{lemma}[theorem]{Lemma}

\newtheorem{proposition}[theorem]{Proposition}

\newtheorem{definition}[theorem]{Definition}
\newtheorem{remark}[theorem]{Remark}

\newcommand{\RR}{\mathbb{R}}

\newcommand{\NN}{\mathbb{N}}
\newcommand{\ZZ}{\mathbb{Z}}

\newcommand{\PP}{\mathbb{P}}
\newcommand{\EE}{\mathbb{E}}

\newcommand{\sN}{\mathcal{N}}

\newcommand{\sS}{\mathcal{S}}

\newcommand{\one}{{\bf 1}}
\newcommand{\sss}{\mathbf{s}}
\newcommand{\su}{\mathbf{u}}

\newcommand{\Tr}{\mathsf{Tr}}

\definecolor{AfonsoBlue}{RGB}{30,65,123}
\definecolor{weird}{RGB}{200, 55, 130}

\title{The Spectral Norm of Random Lifts of Matrices
}
\date{}

\usepackage{authblk}
\author[1]{Afonso S.\ Bandeira\thanks{Email: \textit{bandeira@math.ethz.ch}. Part of this work was done while ASB was with the Department of Mathematics at the Courant Institute of Mathematical Sciences, and the Center for Data Science, at New York University; and partially supported by NSF grants DMS-1712730 and DMS-1719545, and by a grant from the Sloan Foundation.}}
\author[2]{Yunzi Ding\thanks{Email: \textit{yding@nyu.edu}. Partially supported by NSF grant DMS-1712730.}}

\affil[1]{Department of Mathematics, ETH Zurich, Switzerland}
\affil[2]{Department of Mathematics, Courant Institute of Mathematical Sciences, New York University, USA}

\begin{document}

\maketitle
\thispagestyle{empty}

\begin{abstract}
    We study the spectral norm of random lifts of matrices. Given an $n\times n$ symmetric matrix $A$, and a centered distribution $\pi$ on $k\times k\ (k\ge 2)$ symmetric matrices with spectral norm at most $1$, let the \textit{matrix random lift} $A^{(k,\pi)}$ be the random symmetric $kn\times kn$ matrix $(A_{ij}X_{ij})_{1\le i < j \le n}$, where $X_{ij}$ are independent samples from $\pi$. We prove that
    \[\EE \|A^{(k,\pi)}\|\lesssim \max_{i}\sqrt{\sum_j A_{ij}^2}+\max_{ij}|A_{ij}|\sqrt{\log (kn)}.\]
   This result can be viewed as an extension of existing spectral bounds on random matrices with independent entries, providing further instances where the multiplicative $\sqrt{\log n}$ factor in the Non-Commutative Khintchine inequality can be removed. 
   
   As a direct application of our result, we 
   prove an upper bound of $2(1+\epsilon)\sqrt{\Delta}+O(\sqrt{\log(kn)})$ on the new eigenvalues for random $k$-lifts of a fixed $G = (V,E)$ with $|V| = n$ and maximum degree $\Delta$, compared to the previous result of $O(\sqrt{\Delta\log(kn)})$ by Oliveira~\cite{oliveira-k-lifts} and the recent breakthrough by Bordenave and Collins~\cite{bc-eigenvalues} which gives $2\sqrt{\Delta-1} + o(1)$ as $k\rightarrow\infty$ for $\Delta$-regular graph $G$.
\end{abstract}

\newpage

\section{Introduction}
\subsection{The Non-Commutative Khintchine inequality}
The Non-Commutative Khintchine (NCK) inequality, originally introduced by Lust-Piquard and Pisier~\cite{pisier-book}, is one of the simplest tools for understanding the spectrum of matrix series, namely  
\begin{equation}\label{matrix-series}
    X = \sum_{i = 1}^N \gamma_i A_i
\end{equation}
where $A_i\ (i = 1,2,\dots,N)$ are $n\times n$ real symmetric matrices and $\gamma_i\ (i = 1,2,\dots,N)$ are i.i.d.\ random variables, usually assumed gaussian or Rademacher. The inequality is stated as follows.
\begin{theorem}[Non-Commutative Khintchine (NCK) inequality] 
Let $A_1,A_2,\dots, A_N$ be $n\times n$ symmetric matrices and $\gamma_1,\gamma_2,\dots,\gamma_N$ be i.i.d.\ $\sN(0,1)$ random variables, then
\begin{equation}\label{nck}
    \EE \left\|\sum_{i = 1}^N\gamma_i A_i\right\|\le \sigma\sqrt{2+2\log(2n)},
    \quad \text{where}\  \sigma := \left(\left\|\sum_{i = 1}^N A_i^2\right\|\right)^{1/2}.
\end{equation}
\end{theorem}

\noindent
The NCK inequality and other phenomena of \textit{matrix concentration} have been proven under various settings and extensively studied in~\cite{oliveira-concentration,tropp-tail-bounds,tropp-intro}. One particularly important application of matrix concentration is on the spectra of random matrices with independent entries. These random matrices can be represented as matrix series upon a direct entry-wise decomposition, as we show below.

\subsection{Random matrices with independent entries}
The study of random matrices with independent entries traces back to the seminal work by Wigner~\cite{wigner}. For Wigner matrices (real symmetric or Hermitian random matrices with independent mean-zero and unit variance entries), a long line of work has established a comprehensive understanding towards its spectral properties over the past decades (see, for example,~\cite{furedi-eigenvalues,bai-yin-wigner,anderson-book,tao-book}). One of the most important results is the Wigner semicircle law: for $n\times n$ Wigner matrix $X$, $\EE\|X\|/\sqrt{n}\rightarrow 2$ and the spectrum of $X$ converges to the semicircle $\frac{1}{2\pi}\sqrt{4-x^2}\one_{\{-2\le x\le 2\}}$ as $n\rightarrow\infty$. 

Random matrices with different variances on each of the independent entries, for instance real symmetric $X\in \RR^{n\times n}$ with $X_{ij}\sim \sN(0,b_{ij}^2)$ for $1\le i\le j\le n$, have also been studied~\cite{davidson-random-matrix,rudelson-non-asymptotic,vershynin-non-asymptotic}. With the NCK inequality, the following estimate can be obtained:
\begin{equation}\label{matrix-nck-bound}
    \EE \|X\|\lesssim \sigma\sqrt{\log n}, \quad \text{where}\ \sigma:= \max_{i}\sqrt{\sum_j b_{ij}^2}.
\end{equation}
Here $A\lesssim B$ (respectively, $A\gtrsim B$) refer to $A\le CB$ (respectively, $A\ge CB$) for some absolute positive constant $C$. The definition of $\sigma$ is consistent with~\eqref{nck} upon writing $X$ as the following matrix series:
\begin{equation}
    \label{matrix-series}
    X = \sum_{i = 1}^n \gamma_{ii}b_{ii}E_{ii} + 
\sum_{1\le i < j\le n}\gamma_{ij} b_{ij}(E_{ij} + E_{ji}).
\end{equation}
Here $E_{ij} := e_i e_j^\top$. One may immediately notice that the bound~\eqref{matrix-nck-bound} is not sharp for Wigner matrices with i.i.d.\ standard gaussian entries, since it gives $\EE \|X\|\lesssim \sqrt{n\log n}$ rather than $\EE \|X\|\sim \sqrt{n}$. In fact, a recent improvement for matrices with independent entries~\cite{bandeira-sharp-bound} yields
\begin{equation}\label{bvh-sharp-bound}
    \EE \|X\|\lesssim \sigma + \sigma_*\sqrt{\log n}, \quad \text{where}\ \sigma_* := \max_{ij}|b_{ij}|.
\end{equation}
This upper bound is sharp as a matching lower bound $\EE \|X\|\gtrsim \sigma + \sigma_*\sqrt{\log n}$ is also given in~\cite{bandeira-sharp-bound} under mild assumptions on the $b_{ij}$'s. 
Further refinements, that hold for general $b_{ij}$'s, have been recently obtained~\cite{van-handel-conj,latala-bound-conj, benaych-georges-radii}.

\subsection{Improving the NCK inequality}
The gap between the NCK bound~\eqref{matrix-nck-bound} and its improvement~\eqref{bvh-sharp-bound} demonstrates the sub-optimality of the NCK inequality in some settings. In fact, many improvements to bounds obtained via the NCK inequality are known under various settings. Seginer~\cite{seginer-bound} improved the NCK bound for random matrices with independent uniformly bounded entries. Exploiting the non-commutativity among the $A_i$'s, Tropp~\cite{tropp-nck-improve} proved the following upper bound for the series~\eqref{matrix-series}, which improved the multiplicative factor on $\sigma$ from $\sqrt{\log n}$ to $\sqrt[4]{\log n}$: 
\[\EE \|X\| \lesssim \sigma\sqrt[4]{\log n} + \omega \sqrt{\log n}\]
where the alignment parameter $\omega$ is defined as
\[
\omega := \max_{Q_1,Q_2,Q_3\in U_n}\left\| \sum_{i,j = 1}^n  A_iQ_1A_jQ_2A_iQ_3A_j\right\|^{1/4},
\]
 here $U_n$ denotes the group of $n\times n$ unitary matrices (the paper~\cite{tropp-nck-improve} considers Hermitian $A_i$'s, while in this paper we focus on the real symmetric setting). In fact, a bound which replaces the multiplicative $\sqrt{\log n}$ factor in the NCK inequality by an additive factor has been hypothesized in many different forms~\cite{tropp-tail-bounds,10l42op,bandeira-sharp-bound,van-handel-conj,latala-bound-conj,tropp-nck-improve}. For the matrix series~\eqref{matrix-series}, define the ``weak variance" as
 \begin{equation}\label{def-sigma-star}
    \sigma_* = \left(\max_{\|v\| = 1} \sum_{i = 1}^N (v^\top A_i v)^2\right)^{1/2},
\end{equation}
a possible improvement to Theorem~\ref{nck} could be
\begin{equation}\label{nck-improve-ineq}
    \EE \left\|\sum_{i = 1}^N\gamma_i A_i\right\|\le C(\sigma+\sigma_*\sqrt{\log n}).
\end{equation}
Note that $\sigma_*\le\sigma$ by a simple application of the Cauchy-Schwarz inequality. For random matrices with independent $X_{ij}\sim \sN(0,b_{ij}^2)$ for $1\le i\le j\le n$, upon writing it as a matrix series as in~\eqref{matrix-series}, we have
\[X = \sum_{i = 1}^{N} \gamma_i A_i, \quad {\rm where}\ \  \{A_i\}_{i = 1}^N = \{b_{ii}E_{ii}\}_{i = 1}^n \cup \{b_{ij}(E_{ij} + E_{ji})\}_{1\le i < j\le n}.\]
It is not hard to show that
\[\max_{\|v\| = 1} \left(\sum_{i = 1}^N (v^\top A_i v)^2\right)^{1/2} \asymp \max_{ij}|b_{ij}|,\]
thus $\sigma_*$ defined in~\eqref{def-sigma-star} is consistent with the quantity defined in~\eqref{bvh-sharp-bound} in that they differ only by a multiplicative constant, and the proposed improvement~\eqref{nck-improve-ineq} is indeed true in the case of random matrices with independent entries due to~\cite{bandeira-sharp-bound}.\\

In this paper, we show another class of examples, in which we improve the bound given by the NCK inequality~\eqref{nck} by replacing the multiplicative $\log$ factor with an additive factor, as is the case in the conjectured bound~\eqref{nck-improve-ineq}. In the following context, all the matrices we consider are real. As an extension to random matrices with independent entries, we consider the operation of \textit{matrix lifts}, in which each entry of an underlying deterministic matrix is replaced by the product of itself and a random $k\times k$ matrix, as described in the following definition. 
\begin{definition}[Matrix lifts]\label{def-matrix-lift}
 Let $A$ be an $n\times n$ symmetric matrix with zero diagonal entries, and $\pi$ be a measure supported on $k\times k$ matrices. Define the $(k,\pi)$ lift of $A$, denoted $A^{  (k,\pi)}$, as follows: 
\begin{itemize}
    \item Draw i.i.d.\ samples $\{\Pi_{ij}\}_{1\le i < j\le n}$ from $\pi$; for $1\le i < j\le n$, denote $\Pi_{ji}:= \Pi_{ij}^\top$;
    \item For $1\le i\le n$, denote $\Pi_{ii} := 0_k$;
    \item For all $1\le i, j \le n$, replace $A_{ij}$ with the matrix $A_{ij}\Pi_{ij}$.
\end{itemize}
The resulting matrix is a $kn\times kn$ symmetric random matrix, which can be written as
\begin{equation}
\label{def-matrix-lift-good}
    A^{ (k,\pi)} = \sum_{1\le i,j\le n}A_{ij}(E_{ij}\otimes \Pi_{ij})
\end{equation}
where the symbol ``$\otimes$" on the RHS denotes the Kronecker product of matrices.
\end{definition}
\noindent
The main theorem of this paper is the following bound:
\begin{theorem}\label{thm-group}
Let $A$ be a symmetric $n\times n$ matrix ($n\ge 2$) with zero diagonal entries. Suppose $\pi$ is a centered measure supported on $k\times k$ matrices with spectral norm at most $1$. Then there exists a universal constant $C$, such that for any $\epsilon \in (0, 1/2]$, 
\begin{equation}\label{eqn-thm-group}
    \EE \|A^{  (k,\pi)}\| \le 2(1+\epsilon)\sigma + \frac{C}{\sqrt{\log(1+\epsilon)}}\sigma_* \sqrt{\log(kn)}.
\end{equation}
where
\begin{equation*}
    \sigma := \max_{i}\sqrt{\sum_j A_{ij}^2},\quad \sigma_* := \max_{ij}|A_{ij}|.
\end{equation*}
\end{theorem}
\noindent
Note that Definition~\ref{def-matrix-lift} and Theorem~\ref{thm-group} only apply to base matrices $A$ with zero diagonal entries. A base matrix with possibly non-zero diagonal entries can be handled by splitting it in its diagonal and non-diagonal parts and using triangular inequality in the result random matrices. 
\begin{remark}
Upon taking $A = \{b_{ij}\},k = 1$ and $\pi = {\rm Uniform}\{\pm 1\}$, Definition~\ref{def-matrix-lift} and Theorem~\ref{thm-group} include as a special case the real symmetric random matrix $X\in \RR^{n\times n}$ with $X_{ij} = \epsilon_{ij}b_{ij}$, where $\{b_{ij}\}$ are given and $\epsilon_{ij}$ are independent Rademacher random variables for $1\le i\le j\le n$, that is, $\PP\left[\epsilon_{ij} = \pm 1\right] = 1/2$. Since~\cite{bandeira-sharp-bound} showed that the bound $O(\sigma + \sigma_*\sqrt{\log n})$ captures the optimal scaling of $\EE\|X\|$ with respect to $\sigma$ and $\sigma_* \sqrt{\log n}$ and is in general unimprovable, this implies the same for our bound~\eqref{eqn-thm-group} on $\EE \|A^{  (k,\pi)}\|$. However, Theorem~\ref{thm-group} does not directly imply the bound~\eqref{bvh-sharp-bound}, since gaussian random variables are not compactly supported. 
\end{remark}
\noindent
Besides the $k = 1$ case, Theorem~\ref{thm-group} is also interesting with natural choices such as $\pi$ being the Haar measure on the orthogonal group $O(k)$ or special orthogonal group $SO(k)$. One particular application is an estimate on the spectrum of random lifts of graphs, which we discuss below.

\subsection{Application: random lifts of graphs}
Given an undirected graph $G = (V,E)$ and an integer $k\ge 2$, the random $k$-lift of $G$, denoted $G^{(k)}$, is obtained by replacing each vertex $v\in V$ by $k$ new vertices, and each edge $e = (v_1,v_2)$ by a random $k\times k$ bipartite matching between the $k$ new vertices corresponding to $v_1$ and those corresponding to $v_2$. Here ``random'' refers to a uniform choice on all $k!$ possible bipartite matchings. We denote $A$ and $A^{ (k)}$ the adjacency matrix of $G$ and $G^{ (k)}$, respectively. 

Previous studies on the $k$-lifts of graphs, under the setting of fixed $G$ and $k\rightarrow\infty$, have revealed many properties of the resulting random graph, such as connectivity~\cite{amit-k-lifts-1}, chromatic number~\cite{amit-k-lifts-chromatic}, edge expansion~\cite{amit-k-lifts-2} and the existence of perfect matching~\cite{linial-k-lifts-matching}. The spectrum of random $k$-lifts, namely the new spectrum introduced in the lifting process
\begin{equation}
    \label{def-spec-k-lift}
    \max_{\eta\in {\rm spec}(A^{(k)})\backslash {\rm spec}(A)} |\eta| = \|A^{ (k)}-\EE A^{ (k)}\|
\end{equation}
was studied by Friedman~\cite{friedman-k-lifts-expander} via the trace method, who showed that with a random $d$-regular graph as the base graph, as $k\rightarrow\infty$,~\eqref{def-spec-k-lift} is $O(d^{3/4})$ with high probability. He also conjectured the tight bound $2\sqrt{d-1} + o(1)$. The high probability upper bound on~\eqref{def-spec-k-lift} was improved by Linial and Puder to $O(d^{2/3})$ in~\cite{linial-k-lifts-spectral}, then by Lubetzky, Sudakov and Vu~\cite{lubetzky-k-lifts-spectral} to $O(\sqrt{d}\log d)$ in the case that the second eigenvalue of the base graph is $O(\sqrt{d})$. Later, Addario-Berry and Griffiths~\cite{ag-lifts-bound} and Puder~\cite{puder-lifts-bound} proved that~\eqref{def-spec-k-lift} is $O(\sqrt{d})$, the latter giving $2\sqrt{d-1} +O(1)$ as an upper bound. Since then, various extensions or alternative proofs of the bound on~\eqref{def-spec-k-lift} of the scale $O(\sqrt{d})$ (some under slightly different settings) have been carried out with different combinatorial and probabilistic techniques, for example, in~\cite{friedman-lifts,bordenave-k-lifts-spectral-1,bordenave-k-lifts-spectral-2,agarwal-lifts-bound}. 

We should notice that the above line of work adopted the asymptotic regime $k\rightarrow\infty$ and the setting that the base graph is taken randomly over all $d$-regular graphs on $n$ vertices. In fact, in the case that the base graph is a fixed $d$-regular graph and $k\in\NN_+$ is fixed,~\eqref{def-spec-k-lift} is not always upper bounded by $O(\sqrt{d})$. As a counterexample (see~\cite{bandeira-sharp-bound}, Remark~{4.8}): consider $G$ the union of $n/s$ cliques of $s$ vertices each, with no edges between different cliques; here $s = \lceil\sqrt{\log n}\rceil$, and we assume that $n/s$ is an integer for simplicity. Seginer~\cite{seginer-bound} showed that
\[\EE \|A^{(2)}-\EE A^{(2)}\| \sim \sqrt{\log n},\]
whereas the $O(\sqrt{d})$ bound would incorrectly predict that LHS is $O(\log^{1/4}(n))$.

Another line of work considers a fixed base graph $G$ with maximum degree $\Delta$ without assuming its randomness. Making use of matrix concentration, Oliveira~\cite{oliveira-k-lifts} obtained a high probability upper bound of $O(\sqrt{\Delta \log(kn)})$ on~\eqref{def-spec-k-lift}. The most recent advancement by Bordenave and Collins~\cite{bc-eigenvalues} considered the $k$-lifts problem under a much more general framework, and proved that~\eqref{def-spec-k-lift} is $2\sqrt{d-1}+o(1)$ for any $d$-regular base graph $G$ as $k\rightarrow\infty$, finally settling Friedman's conjecture even without assuming the randomness of the base graph.

In what follows, we manage to improve the bound in~\cite{oliveira-k-lifts} by removing the multiplicative factor $\sqrt{\log(kn)}$, replacing it with an additive factor. We also improve the constant factor before $\sqrt{\Delta}$ down to $2$, in consistence with Friedman's theorem and~\cite{bc-eigenvalues}. In the large $k$ limit, our bound is weaker than~\cite{bc-eigenvalues} by an additive $\sqrt{\log(kn)}$ factor. For $k=2$, an additive $\sqrt{\log n}$ factor is needed as illustrated by a counterexample due to Seginer~\cite{seginer-bound}. However, our result does not capture the correct dependence on $k$, namely concentration arising from large $k$. Note that we are only using a slight modification of the moment method, compared to the sophisticated combinatorial technique in~\cite{bc-eigenvalues}.

\begin{theorem}\label{thm-k-lift}
Let $A$ be the adjacency matrix of $G = (V,E)$ with $|V| = n$ and ${\rm maxdeg}(G) = \Delta$, and $A^{ (k)}$ be the corresponding random $k$-lift. Then there exists a universal constant $C$, such that for any $\epsilon \in (0, 1/2]$
\begin{equation}
    \EE \|A^{ (k)}-\EE A^{ (k)}\| \le 2(1+\epsilon)\sqrt{\Delta} + \frac{C}{\sqrt{\log(1+\epsilon)}}\sqrt{\log(kn)}.
\end{equation}
\end{theorem}

\noindent
 Our bound is essentially $2(1+\epsilon)\sqrt{\Delta}$ as long as $\Delta \gg \log(kn)$, i.e. the base graph $G$ is not too sparse. The proof of Theorem~\ref{thm-k-lift} will follow from our main result, Theorem~\ref{thm-group}. 

\subsection*{Notation} In this paper, for positive quantities $A$ and $B$, $A\lesssim B$ and $A\gtrsim B$ respectively refer to $A\le CB$ and $A\ge CB$ for some absolute positive constant $C$. For $x\in \RR$, $\lceil x\rceil$ denotes the minimum integer that is larger than or equal to $x$.

\section{Proof of main results}
\label{sec:proof}

In this section, we carry out the proof of Theorems~\ref{thm-group} and~\ref{thm-k-lift}. We begin with the following comparison argument which links  $A^{ (k,\pi)}$ to an auxiliary Wigner matrix. This argument is a modification of Proposition 2.1 in~\cite{bandeira-sharp-bound}, and the auxiliary matrix $Y_r$ is same as in the proof of Theorem 4.8 in~\cite{latala-bound-conj}.

\begin{proposition}\label{compare}
 Let $Y_r$ be the $r\times r$ symmetric matrix with zero diagonal and 
 \begin{equation*}
     Y_{ij} = \left\{
     \begin{aligned}
     \sqrt{3}, &\quad w.p.\ \frac{1}{4}\\
     -\frac{1}{\sqrt{3}},&\quad w.p.\ \frac{3}{4}
     \end{aligned}
     \right.
 \end{equation*}
 independently for all $1\le i < j\le r$. Under the setting of Theorem~\ref{thm-group}, suppose $\sigma_* \le 1$, then for every $p\in\NN_+$ there holds
 \begin{equation*}
     \EE\Tr \left[(A^{ (k,\pi)})^{2p}\right] \le \frac{kn}{\lceil\sigma^2\rceil+p} \EE\Tr\left[ Y_{\lceil\sigma^2\rceil+p}^{2p}\right].
 \end{equation*}
\end{proposition}

\noindent
To carry out the proof of Proposition~\ref{compare}, we start with a set of standard notations adopted from~\cite{furedi-eigenvalues} and~\cite{bandeira-sharp-bound}. Following the representation~\eqref{def-matrix-lift-good}, a direct expansion of $(A^{ (k,\pi)})^{2p}$ yields
\begin{equation}\label{expansion}
      \EE\Tr \left[(A^{ (k,\pi)})^{2p}\right] 
        = \sum_{u_1,u_2,\dots,u_{2p}\in [n]}\left(\prod_{j = 1}^{2p}A_{u_j u_{j+1}}\right) 
        \EE\Tr \left[\prod_{j = 1}^{2p} \Pi_{u_j u_{j+1}}\right].
\end{equation}
Let $G_n = ([n],E_n)$ be the complete graph on $n$ points. A \textit{cycle} $u_1\rightarrow u_2\rightarrow \dots \rightarrow u_{2p}\rightarrow u_1$ of length $2p$, where $u_i \in [n]$ for all $1\le i\le 2p$ ($u_{2p+1} := u_1$), is identified as $\su = (u_1,\dots,u_{2p})\in [n]^{2p}$. Since $\EE \left[\Pi_{ij}\right] = 0$ for any $1\le i,j\le n$, in the sum of~\eqref{expansion} we only need to consider cycles with each edge appearing at least twice. 

We call the \textit{shape} of a cycle $\su$, denoted $\sss(\su)$, a relabeling of the vertices in the order of their appearance.  For example, the shape of $\su = (4,7,2,7,9,4,5,4)$ is $(1,2,3,2,4,1,5,1)$. The following set is a collection of all cycles of shapes that contribute to the sum in~\eqref{expansion}:
\[\sS_{2p} := \{\sss(\su) : \su\ \text{is a cycle of length $2p$ with each edge appearing at least twice}\}.\]
For the sake of convenience, we also define the set of cycles with fixed shape and starting point as
\[\Gamma_{\sss,u} := \{\su\in [n]^{2p}: \sss(\su) = \sss, u_1 = u\}.\]
The \textit{span} of a shape $\sss$, denoted by $m(\sss)$, is the largest index in its representation, also the number of distinct vertices any cycle of shape $\sss$ visits. A direct observation is $m(\sss)\le p+1$ for any $\sss\in\sS_{2p}$.

\begin{proof}[Proof of Proposition~\ref{compare}]
Following the expansion~\eqref{expansion} we have
\begin{equation}\label{upbound-X}
    \begin{aligned}
     \EE\Tr \left[(A^{ (k,\pi)})^{2p}\right] 
        &= \sum_{\sss\in \sS_{2p}}\sum_{u\in [n]}\sum_{\su\in \Gamma_{\sss,u}}\left(\prod_{j = 1}^{2p}A_{u_j u_{j+1}}\right) 
        \EE\Tr \left[\prod_{j = 1}^{2p} \Pi_{u_j u_{j+1}}\right]\\
        &\le k\sum_{\sss\in \sS_{2p}}\sum_{u\in [n]}\sum_{\su\in \Gamma_{\sss,u}}\left(\prod_{j = 1}^{2p}|A_{u_j u_{j+1}}|\right) \\
        &\le kn\sum_{\sss\in \sS_{2p}}\sigma^{2(m(\sss)-1)}
    \end{aligned}
\end{equation}
where the first inequality follows from $\|\prod_{j = 1}^{2p}\Pi_{u_j u_{j+1}}\| \le 1$ and therefore $\Tr\left[\prod_{j = 1}^{2p}\Pi_{u_j u_{j+1}}\right]\le k$; and the second inequality owes to the fact that, under $\sigma_* \le 1$, for any $u\in [n]$ and $\sss\in \sS_{2p}$, Lemma 2.5 in~\cite{bandeira-sharp-bound} gives
\begin{equation*}\label{upb-prod}
    \sum_{\su\in\Gamma_{\sss,u}}\left(\prod_{j = 1}^{2p}|A_{u_j u_{j+1}}|\right)\le \sigma^{2(m(\sss)-1)}.
\end{equation*}
Meanwhile, for any positive integer $r > p$, for the auxiliary random matrix $Y_r$ we have
\begin{equation*}\label{lowbound-Y}
    \begin{aligned}
     \EE\Tr \left[Y_r^{2p}\right] 
        &= \sum_{\su\in [r]^{2p}}
        \EE\Tr \left[\prod_{j = 1}^{2p}Y_{u_j u_{j+1}}\right]\\
        &\ge \sum_{\sss\in \sS_{2p}}\sum_{u\in [r]}\sum_{\su\in \Gamma_{\sss,u}}\EE\Tr \left[\prod_{j = 1}^{2p}Y_{u_j u_{j+1}}\right] \\
        &= \sum_{\sss\in \sS_{2p}}|\{\su\in [r]^{2p}: \sss(\su) = \sss\}|\cdot \EE\Tr \left[\prod_{j = 1}^{2p}Y_{u_j u_{j+1}}\right]\\
        &\ge \sum_{\sss\in \sS_{2p}} r(r-1)\cdots (r-m(\sss)+1)\cdot 1
    \end{aligned}
\end{equation*}
The last inequality follows from the observation that $\EE [Y_{ij}^{m}]\ge 1$ for all $m\ge 2$. Now choosing $r =\lceil\sigma^2\rceil +p$, noting that $m(\sss) \le p+1$ for all $\sss\in\sS_{2p}$, we have
\begin{equation}\label{lowbound-final-Y}
    \EE\Tr \left[Y_{\lceil\sigma^2\rceil+p}^{2p}\right] \ge (\lceil\sigma^2\rceil+p)\sum_{\sss\in \sS_{2p}} \sigma^{2(m(\sss)-1)}.
\end{equation}
Comparing~\eqref{upbound-X} with~\eqref{lowbound-final-Y} yields the result.
\end{proof}
\noindent
The following lemma gives an upper bound on the moments of the auxiliary random matrix $Y_r$.
\begin{lemma}\label{moment-upbound-y}
For $Y_r$ defined in Proposition~\ref{compare}, there exists an absolute constant $C$, such that for any positive integer $p\ge 2$, there holds
\begin{equation}
    \EE\Tr \left[Y_r^{2p}\right] \le r(2\sqrt{r}+C\sqrt{p})^{2p}.
\end{equation}
\end{lemma}
\noindent
\begin{proof}
Since
\[\EE\Tr \left[Y_r^{2p}\right] \le r\EE\left[\|Y_r\|^{2p}\right], \]
we only need to show that there exists an absolute constant $C$, such that for $p\ge 2$,
\begin{equation}\label{eq-reduced-Yr}
    \EE \left[\|Y_r\|^{2p}\right] \le (2\sqrt{r}+C\sqrt{p})^{2p}.
\end{equation}
The proof of~\eqref{eq-reduced-Yr} is contained in the proof of Theorem 4.8 in~\cite{latala-bound-conj}, so we do not repeat it here. The main steps of the proof are a norm bound for Wigner matrices with non-symmetrically distributed entries followed by Talagrand's concentration inequality.
\end{proof}

\begin{proof}[Proof of Theorem~\ref{thm-group}]
By Proposition~\ref{compare} and Lemma~\ref{moment-upbound-y}, assuming $\sigma_* \le 1$, we know that for any positive integer $p\ge 2$,
\begin{align*}
    \EE \|A^{ (k,\pi)}\| &\le \left(\EE \Tr \left[ (A^{ (k,\pi)})^{2p}\right] \right)^{1/2p}\\
    &\le \left(\frac{kn}{\lceil \sigma^2 \rceil+p} \EE\Tr \left[ Y_{\lceil\sigma^2\rceil+p}^{2p}\right]\right)^{1/2p}\\
    &\le (kn)^{1/2p}\left(2\sqrt{\lceil \sigma^2 \rceil+p}+C\sqrt{p}\right)
\end{align*}
If $kn\ge 3$, for $\alpha \ge 1$ choosing $p = \lceil \alpha\log(kn) \rceil \ge 2$ yields
\begin{align*}
    \EE \|A^{ (k,\pi)}\| &\le e^{1/2\alpha}\left(2\sqrt{\lceil\sigma^2\rceil+\lceil\alpha \log(kn) \rceil}+C\sqrt{\lceil\alpha \log(kn) \rceil}\right)\\
    &\le e^{1/2\alpha}\left(2\sigma+2\sqrt{\alpha\log(kn)+2}+C\sqrt{\alpha\log(kn)+1 }\right).
\end{align*}
\noindent
Denote $e^{1/2\alpha} = 1+\epsilon$. Since $n\ge 2$, $k\ge 1$, and $\epsilon\le 1/2$ implies $\alpha \ge 1$, we have $2 < 3\log 2 \le 3\alpha\log(kn)$, so
\[\EE \|A^{ (k,\pi)}\| \le 2(1+\epsilon)\sigma + (1+\epsilon)(4+2C) \sqrt{\frac{\log(kn)}{2\log(1+\epsilon)}}.\]
\noindent
The remaining case $kn < 3$ can only happen when $n = 2$ and $k = 1$. In this case $\pi$ is supported on $[-1, 1]$, and we can directly estimate
\[\EE \|A^{ (k,\pi)}\| \le \EE\|A^{ (k,\pi)}\|_F \le 2 < 2\sqrt{2\log(kn)}.\]
\end{proof}
\noindent
The spectral bound of random $k$-lifts of graphs follows as an immediate corollary.

\begin{proof}[Proof of Theorem~\ref{thm-k-lift}]
Denote ${\rm Perm}(k)$ the collection of all $k\times k$ permutation matrices, and
\[G_k := \{\Pi-\frac{1}{k}J_k : \Pi\in {\rm Perm}(k)\}\]
where $J_k$ is the $k\times k$ matrix with all entries $1$. It is easy to verify that $\|X\|\le 1$ for any $X\in G_k$. Moreover, the adjacency matrix $A$ has $\sigma^2 \le \Delta$ and $\sigma_* \le 1$ by definition. Thus it follows from Theorem~\ref{thm-group} that
\begin{align*}
    \EE \|A^{ (k)}-\EE A^{ (k)}\|&= \EE \|A^{  (k,{\rm Unif}(G_k))}\|\\
    & \le 2(1+\epsilon)\sqrt{\Delta} + \frac{C}{\sqrt{\log(1+\epsilon)}}\sqrt{\log(kn)}.
\end{align*}
\end{proof}
\begin{remark}
In the above proof, we applied Theorem~\ref{thm-group} on $\pi = {\rm Unif}(G_k)$, where $G_k$ is the centered version of ${\rm Perm}(k)$. One may expect that, under the setting of Theorem~\ref{thm-group} without assuming $\pi$ is centered, there still holds
\begin{equation}\label{thm-group-wrong}
    \EE \|A^{  (k,\pi)}-\EE A^{  (k,\pi)}\| \le C\left(\sigma + \sigma_*\sqrt{\log(kn)}\right).
\end{equation}
Though we do not have a counterexample for~\eqref{thm-group-wrong}, we must point out that~\eqref{thm-group-wrong} only follows from Theorem~\ref{thm-group} when $\|X - \EE_\pi X\| \le 1$ for every $X\in {\rm supp}(\pi)$. 
\end{remark}
We note that the proof of Theorem~\ref{thm-group} is not exploiting any potential structure of the ``lifting matrices" $\Pi_{ij}$. In fact, this may explain why Theorem~\ref{thm-k-lift} is worse than the result in~\cite{bc-eigenvalues} by an additive $\sqrt{\log(kn)}$ factor in the large $k$ limit for $d$-regular base graphs. One may be able to obtain a stronger result, for instance $\EE\|A^{(k, \pi)}-\EE A^{(k,\pi)}\| \le 2\sqrt{\Delta} + o_k(1)$, with a more careful analysis considering that $\Pi_{ij}$ are permutation matrices.

\section*{Acknowledgments}
\addcontentsline{toc}{section}{Acknowledgments}
We are grateful to Ramon van Handel for comments on an early version of the paper, in particular for pointing us to the latest results on graph $k$-lifts in~\cite{bc-eigenvalues}, for directing us to the proof of Theorem 4.8 in~\cite{latala-bound-conj} which allowed us to improve the constant factor before $\sigma$ in Theorem~\ref{thm-group} to $2$, and for making us aware of recent efforts to improve the NCK inequality under more general settings. We would also like to thank Jiedong Jiang, Eyal Lubetzky, Ruedi Suter and Joel Tropp for helpful discussions. 

\bibliographystyle{alpha}
\bibliography{main}

\end{document}